\documentclass{amsart}
\usepackage{amsthm, amssymb, amsfonts, amscd}
\usepackage{graphicx}
\usepackage{tikz-cd}
\usepackage{verbatim}
\usepackage{enumitem}
\usepackage{mathrsfs, mathtools}
\usepackage{xcolor}
\usepackage[colorlinks = true,
linkcolor = black,
urlcolor  = black,
citecolor = black]{hyperref}
\usepackage{cleveref}
\usepackage[margin=1in]{geometry}
\usepackage{setspace}
\usepackage{diagbox}
\allowdisplaybreaks
\theoremstyle{definition} 
\newtheorem{thm}{Theorem}[section]

\newtheorem{lem}[thm]{Lemma}
\newtheorem{conj}[thm]{Conjecture}

\newtheorem{rmk}[thm]{Remark}

\newtheorem{setup}[thm]{Setup}

\newcommand{\bb}{\mathbb}
\newcommand{\bN}{\bb{N}}

\newcommand{\bC}{\bb{C}}
\newcommand{\bF}{\bb{F}}
\newcommand{\fp}{\bF_p}
\newcommand{\ol}{\overline}
\newcommand{\olfp}{\ol{\bF}_p}
\newcommand{\Prep}{\mathrm{Prep}}
\newcommand{\ep}{\epsilon}
\newcommand{\de}{\delta}
\newcommand{\al}{\alpha}
\newcommand{\be}{\beta}
\newcommand{\oal}{\ol{\al}}
\newcommand{\lamoa}{\lambda_{\oal}}

\newcommand{\lt}{\left |}
\newcommand{\rt}{\right |}
\newcommand{\mr}{\mathrm}

\begin{document}
	
\title[Simultaneously preperiodic points in characteristic $p$]{On simultaneously preperiodic points for one-parameter families of polynomials in characteristic $p$}
\author{Jungin Lee and GyeongHyeon Nam}
\date{}
\address{J. Lee -- Department of Mathematics, Ajou University, Suwon 16499, Republic of Korea \newline
\indent G. Nam -- Department of Mathematics and Systems Analysis, Aalto University, Otakaari 1, Espoo, 02150, Finland}
\email{jileemath@ajou.ac.kr, gyeonghyeon.nam@aalto.fi}
	
\begin{abstract}
For a field $L$ of characteristic $p$, a polynomial $f \in \overline{\mathbb{F}}_p[x]$ and $\alpha, \beta \in L$, let $\mathrm{Prep}(f;\alpha,\beta)$ be the set of all $\lambda \in \overline{L}$ such that both $\alpha$ and $\beta$ are preperiodic under the action of $f_{\lambda}(x) := f(x) + \lambda$. Ghioca and Hsia proved that for certain families of polynomials, this set is infinite if and only if $f(\alpha)=f(\beta)$ or $\alpha, \beta \in \overline{\mathbb{F}}_p$.
Building on their work, we determine when $\mathrm{Prep}(f;\alpha,\beta)$ is infinite for most of the remaining binomial cases that were left open. 
Specifically, let $f(x)=c_1 x^{d_1} + c_2 x^{d_2} \in \overline{\mathbb{F}}_p[x]$, where $c_i \in \overline{\mathbb{F}}_p^*$, $1 \le d_1 < d_2$ and $d_i=p^{\ell_i}s_i$ with $\ell_i \ge 0$ and $p \nmid s_i$. We prove that if $p^{\ell_2}(s_2-1) < p^{\ell_1}(s_1-1)$, then $\mathrm{Prep}(f;\alpha,\beta)$ is infinite if and only if $f(\alpha)=f(\beta)$ or $\alpha, \beta \in \overline{\mathbb{F}}_p$. 
The key idea of the proof is to use the parameters $\lambda_{\overline{\alpha}} := \overline{\alpha} - f(\overline{\alpha})$ associated to suitable elements $\overline{\alpha} \in \overline{L}$ satisfying $f(\overline{\alpha})=f(\alpha)$. As an application, we extend the work of Asgarli and Ghioca on the colliding orbits problem to binomials satisfying $s_2>1$ and $p^{\ell_2}(s_2-1) < p^{\ell_1}(s_1-1)$.
\end{abstract}
\maketitle

\section{Introduction} \label{Sec1}

\subsection{The unlikely intersection principle} \label{ss:11}

The principle of \emph{unlikely intersections} has played a central role in arithmetic geometry and has inspired major research directions in the past few decades. In particular, it has also generated a rich body of work in arithmetic dynamics. One prominent example is the \emph{dynamical Mordell--Lang conjecture}, which states the following: let $X$ be a quasi-projective variety defined over a field $K$ of characteristic $0$, $\Phi$ be an endomorphism on $X$, $V \subseteq X$ be a closed subvariety and $\al \in X(K)$. Then the set of nonnegative integers $n$ such that $\Phi^n(\al) \in V$ forms a finite union of arithmetic progressions. For further background, we refer to \cite{BGT16}.

Another important question in arithmetic dynamics concerns \emph{simultaneously preperiodic points} for families of polynomials. Baker and DeMarco \cite[Theorem 1.1]{BD11} proved that for an integer $d \ge 2$ and $\al, \be \in \bC$, there are infinitely many parameters $c \in \bC$ such that both $\al$ and $\be$ are preperiodic under iteration of the polynomial $z^d+c$ if and only if $\al^d = \be^d$. 
This result has been extended to arbitrary families of polynomials \cite{BD13, FG18, GHT13, GY18} and certain families of rational maps \cite{DM20, DWY15, GHT15}. 

Traditionally, unlikely intersection problems have been studied mainly in characteristic $0$. More recently, however, there has been growing interest in the positive characteristic setting, where the landscape is more subtle and new conjectures are proposed. For instance, in positive characteristic, see \cite{BGT15, CGSZ21, Ghi19, XY25} for the dynamical Mordell--Lang conjecture, \cite{Ghi24, GH24} for simultaneously preperiodic points for one-parameter families of polynomials and \cite{GS23a, GS23b} for the Zariski dense orbit conjecture.

In this paper, building on the earlier works of Ghioca \cite{Ghi24} and Ghioca--Hsia \cite{GH24}, we resolve certain cases of simultaneously preperiodic points in positive characteristic that lie beyond the scope of previous methods.
Furthermore, motivated by the recent work of Asgarli--Ghioca \cite{AG25}, we also make progress on the \emph{colliding orbits problem} as explained in Section \ref{ss:14}.

\subsection{Previous results} \label{ss:12}
	
The problem of simultaneously preperiodic points in positive characteristic was studied by Ghioca \cite{Ghi24} and Ghioca--Hsia \cite{GH24}.  
Let $L$ be a field of characteristic $p>0$ and $\al, \be \in L$. Let $\ol{L}$ be a fixed algebraic closure and $\olfp$ be the algebraic closure of $\fp$ inside $\ol{L}$. For a polynomial $f \in \olfp[x]$, consider the family of polynomials
$$
f_{\lambda}(x) := f(x) + \lambda \in \ol{L}[x]
$$
parametrized by $\lambda \in \ol{L}$ and let $\Prep(f; \al, \be)$ denote the set of all $\lambda \in \ol{L}$ such that both $\al$ and $\be$ are preperiodic under $f_{\lambda}$. For certain special families of polynomials, the following results are known.

\begin{thm} \label{thm1a}
(\cite[Theorem 1.1]{Ghi24}) Let $f(x)=x^d$ for $d \ge 2$. Then $\lt \Prep(f; \al, \be) \rt = \infty$ if and only if at least one of the following statements holds:
\begin{enumerate}
\item[(a)] $\al, \be \in \olfp$;
\item[(b)] $d = p^{\ell}$ for some positive integer $\ell$ and $\be - \al \in \olfp$;
\item[(c)] $f(\al) = f(\be)$.
\end{enumerate}
\end{thm}

\begin{thm} \label{thm1b}
(\cite[Theorem 1.1]{GH24}) Let $f \in \olfp[x]$ be a polynomial of degree $d \ge 3$. Assume that
$$
f(x) = \sum_{i=1}^{r} c_i x^{d_i},
$$
where $c_i \in \olfp^*$, $1 \le d_1 < \cdots < d_r = d$ and $d_i = p^{\ell_i} s_i$ with $\ell_i \ge 0$ and $p \nmid s_i$. Assume that the inequality
\begin{equation} \label{eqn_1a}
p^{\ell_r} (s_r-1) > \max ( 1, p^{\ell_1} (s_1-1), \ldots, p^{\ell_{r-1}} (s_{r-1}-1) )
\end{equation}
holds. Then $\lt \Prep(f; \al, \be) \rt = \infty$ if and only if at least one of the following statements holds:
\begin{enumerate}
\item[(a)] $\al, \be \in \olfp$;
\item[(b)] $f(\al) = f(\be)$.
\end{enumerate}
\end{thm}

We note that the inequality \eqref{eqn_1a} always holds when $\ell_r=0$. Moreover, by the assumption $d_r=d \ge 3$, at least one of $p^{\ell_r}$ and $s_r$ is greater than $2$ so $p^{\ell_r}(s_r-1) > 1$ unless $s_r=1$. Thus, for $r>1$, the inequality \eqref{eqn_1a} holds if and only if
$$
p^{\ell_r}(s_r-1) > \max( p^{\ell_1} (s_1-1), \ldots, p^{\ell_{r-1}} (s_{r-1}-1) ).
$$

\subsection{Main results} \label{ss:13}
\label{sss:notation and bound}

As our main results focus on the case $r=2$, we work under the following setup.

\begin{setup}\label{setup_f}
Let
$$
f(x)=c_1x^{d_1}+c_2x^{d_2} \in \olfp[x],
$$
where $c_1, c_2 \in \olfp^*$, $1 \le d_1 < d_2$ and $d_i = p^{\ell_i} s_i$ with $\ell_i \ge 0$ and $p \nmid s_i$.
\end{setup}

By Theorem \ref{thm1b}, it suffices to consider the situation where
$$
p^{\ell_2}(s_2-1) \le p^{\ell_1}(s_1-1).
$$
It is unlikely that an analogue of Theorem \ref{thm1b} can be established in this case by just applying the existing methods. Indeed, Ghioca and Hsia \cite[Remark 1.4]{GH24} mentioned that extending Theorem \ref{thm1b} to the case where $s_r=1$ but not all $s_i=1$ for $1 \le i \le r-1$ is very challenging. 

The main contribution of this article is to establish an analogue of Theorem \ref{thm1b} under the condition
$$
p^{\ell_2}(s_2-1) < p^{\ell_1}(s_1-1).
$$
As a special case, this includes the situation where $s_2=1$ and $s_1>1$.

\begin{thm} \label{thm_main4}
Under Setup \ref{setup_f}, assume that 
$$
p^{\ell_2}(s_2-1) < p^{\ell_1}(s_1-1).
$$
Then $\lt \Prep(f; \al, \be) \rt = \infty$ if and only if $\al,\be\in \olfp$ or $f(\al)=f(\be)$.
\end{thm}

By Theorems \ref{thm1b} and \ref{thm_main4}, the only remaining case is when $p^{\ell_2}(s_2-1) = p^{\ell_1}(s_1-1)$.
While we are unable to prove an analogue of Theorem \ref{thm1b} in this case, we can still establish the following weaker result. Here, we exclude the case $s_1=s_2=1$ (equivalently, $f$ is additive) as it will be addressed in Section \ref{ss:add}.

\begin{thm} \label{thm_main5}
Under Setup \ref{setup_f}, assume that $s_1,s_2>1$ and
$$
p^{\ell_2}(s_2-1) = p^{\ell_1}(s_1-1).
$$
If $\lt \Prep(f; \al, \be) \rt = \infty$ and not both $\al$ and $\be$ are contained in $\olfp$, then either $f(\al)=f(\be)$ or
$$
(f(\be)-f(\al))^{p^{\ell_2}-p^{\ell_1}} = -\frac{c_1s_1}{c_2s_2} \in \olfp^*.
$$
\end{thm}

\begin{rmk} \label{rmk_reduction}
In the above theorems, we may assume that $f(x)$ is monic and $L$ is a finite extension of the perfect closure of $\olfp(t)$ containing both $\al$ and $\be$. 
\begin{itemize}
    \item For an element $c \in \olfp$ satisfying $c^{d_2-1}=c_2$, let $\mu_c(x) = cx$ and 
    $$
    g(x) := \mu_c \circ f \circ \mu_c^{-1}(x) = c_1c^{1-d_1}x^{d_1} + x^{d_2} \in \olfp[x].
    $$
    Then $g$ is monic and $\Prep(f; \al, \be)$ is infinite if and only if $\Prep(g; c \al, c \be)$ is infinite. Moreover, If 
    $$
    (f(\be)-f(\al))^{p^{\ell_2}-p^{\ell_1}} = -\frac{c_1s_1}{c_2s_2},
    $$
    then
    $$
    (g(c \be)-g(c \al))^{p^{\ell_2}-p^{\ell_1}} = -\frac{c_1s_1}{c_2s_2} c^{d_2-d_1}
    = -\frac{(c_1c^{1-d_1})s_1}{1 \cdot s_2}.
    $$
    By these observations, we may assume that $f$ is monic.

    \item Now assume that $f$ is monic. If the set $\Prep(f; \al, \be)$ is nonempty, then $\mr{trdeg}_{\fp} \fp(\al, \be) \le 1$ by \cite[Proposition 4.2]{GH24}. Following the argument of \cite[Section 4.1]{GH24} (or \cite[Section 6.1]{Ghi24}), we may assume that $L$ is a finite extension of
    $$
    L_0:= \olfp(t,t^{1/p},t^{1/p^2},\ldots )
    $$
    (the perfect closure of $\olfp(t)$) containing both $\al$ and $\be$.
\end{itemize}
\end{rmk}

As mentioned earlier, our goal is to address the cases left open in \cite{Ghi24} and \cite{GH24}. In those works, the authors only used $\ep_1 = f(\be)-f(\al)$ to prove their results. Our paper introduces $\ep_2$ (and more generally, $\ep_n$), which provides some additional information about the relation between $\al$ and $\be$. However, the use of $\ep_n$ is not sufficient to prove our result.

The essential new ingredient of this paper is the introduction of an element $\oal \in \ol{L}$ satisfying $f(\oal)=f(\al)$. Since $\al$ is preperiodic under $f_{\lambda_{\al}}$ where $\lambda_{\al} := \al - f(\al)$, we have $\widehat{h}_{v, \lambda_{\al}}(\al)=0$ for every place $v$ of $L$. (See Section \ref{ss:21} for the definition of the local canonical height.) 
If the set $\Prep(f; \al, \be)$ is infinite, then Theorem \ref{thm:sameheight} ensures that $\widehat{h}_{v, \lambda}(\al)=\widehat{h}_{v, \lambda}(\be)$ for every $v$ and $\lambda \in \ol{L}$, which implies that $\widehat{h}_{v, \lambda_{\al}}(\be)=0$ for every $v$. This step was crucial in the proofs of Theorems \ref{thm1a} and \ref{thm1b}.

For an element $\oal \in \ol{L}$ satisfying $f(\oal)=f(\al)$, define $\lamoa := \oal - f(\oal)$. Since $f_{\lamoa}(\al) = f_{\lamoa}^2(\al) = \oal$, $\al$ is preperiodic under $f_{\lamoa}$. (Note that $\al$ is not periodic under $f_{\lamoa}$ when $\oal \ne \al$.) Hence, if the set $\Prep(f; \al, \be)$ is infinite, then $\widehat{h}_{v, \lamoa}(\be)=\widehat{h}_{v, \lamoa}(\al)=0$ for every $v$ by Theorem \ref{thm:sameheight}. 
In this setting, the equality $\widehat{h}_{v, \lamoa}(\be)=0$ implies that the absolute value of
$$
\ol{\ep_2}-\ep_2 = f(\ep_1+\oal)-f(\ep_1+\al)
$$
at the place $v$ cannot be too large. By choosing suitable $\oal$'s, we can control the absolute value of $\ol{\ep_2}-\ep_2$, which in turn yields our result.

After some detailed analysis on the valuations of $f(\ep_1+\oal)-f(\ep_1+\al)$, we derive Theorem \ref{thm_main4}. These observations involving $\ep_n$ and $\oal$ are essential for our proofs, as they provide a deeper understanding into the behavior of $\ep_1=f(\be)-f(\al)$.

\subsection{Colliding orbits problem} \label{ss:14}

Recently, Asgarli and Ghioca \cite{AG25} introduced the \emph{colliding orbits problem}, which extends a question studied in \cite{Ghi24}. The problem can be stated as follows. Let $L$ be a field of positive characteristic and let $\al_1,\al_2,\be \in L$. One seeks conditions under which the strict forward orbits of $\al_1$ and $\al_2$ under $f_\lambda$ both contain the target point $\be$ for infinitely many $\lambda \in \ol{L}$.
The case $f(x)=x^d$ ($d \ge 2$) was settled in \cite[Theorem 1.1]{AG25}, and their results suggest the following conjecture.

\begin{conj} \label{conj:colliding}
Let $L$ be a field of characteristic $p>0$, $\al_1,\al_2,\be \in L$ and $f \in \olfp[x]$ be a polynomial of degree $d \ge 2$ which is not additive. Assume that not all of 
$\al_1$, $\al_2$ and $\be$ are contained in $\olfp$.
Then the set
$$
C_f(\al_1, \al_2; \be) := \{\lambda\in \ol{L} : \text{there are } m,n \in \bN \text{ such that } f_\lambda^m(\al_1) = f_\lambda^n(\al_2) = \be \}
$$
is infinite if and only if $f(\al_1)=f(\al_2)$.
\end{conj}

We remark that Theorem \ref{thm1a} plays a crucial role during the proof of \cite[Theorem 1.1]{AG25}. This suggests that our result, Theorem \ref{thm_main4}, can be applied to extend the work of Asgarli and Ghioca to certain families of binomials. Indeed, we prove the following theorem using Theorem \ref{thm_main4}.

\begin{thm} \label{thm1_colliding}
Under Setup \ref{setup_f}, assume that $s_2>1$ and
$$
p^{\ell_2}(s_2-1) < p^{\ell_1}(s_1-1).
$$
Let $\al_1, \al_2, \be \in L$ and assume that not all of $\al_1$, $\al_2$ and $\be$ are contained in $\olfp$. Then $\lt C_f(\al_1, \al_2; \be) \rt  = \infty$ if and only if $f(\al_1)=f(\al_2)$.
\end{thm}

\begin{rmk} \label{rmk:colliding-monic}
For $c \in \olfp$, let $\mu_c(x) = cx$ and 
$g(x) := \mu_c \circ f \circ \mu_c^{-1}(x) \in \olfp[x]$. Then $\lambda \in C_f(\al_1, \al_2; \be)$ if and only if $c \lambda \in C_g(c \al_1, c \al_2; c\be)$ so $C_f(\al_1, \al_2; \be)$ is infinite if and only if $C_g(c \al_1, c \al_2; c\be)$ is infinite.
Under Setup \ref{setup_f}, choosing $c \in \olfp$ satisfying $c^{d-1}=c_2$ guarantees that $g$ is monic.
\end{rmk}

\subsection{Outline of the paper} 

The paper is organized as follows. We recall basic facts on canonical heights in Section \ref{ss:21} following the exposition of \cite{GH24}. In Section \ref{ss:add}, we consider additive polynomials as discussed in \cite[Remark 1.4]{GH24}, and provide an explicit proof. 
In Section \ref{Sec3}, we investigate bounds on the valuations. We establish several results on the local absolute values of the numbers $\ep_n$ (defined in Section \ref{Sec3}) in Section \ref{ss:boundofvalue}. Section \ref{ss:overlinealpha} is devoted to properties of a special element $\oal \in S(f,\al)$, which serves as a key ingredient in the proofs of our main results. 

The proofs of our main results are presented in Sections \ref{Secmain} and \ref{Sec_col}. The proof of Theorem \ref{thm_main4} is divided into two cases: $s_2=1$ (Theorems \ref{thm4c} and \ref{thm4d}) and $s_2>1$ (Theorem \ref{thm4_general}). We prove Theorems \ref{thm4c} and \ref{thm4d} in Section \ref{ss:41}, and Theorem \ref{thm4_general} in Section \ref{ss:42}.
In Section \ref{ss:43}, we prove Theorem \ref{thm_main5} and discuss the limitation of our method. In Section \ref{Sec_col}, we prove Theorem \ref{thm1_colliding}. We prove the forward direction of Theorem \ref{thm1_colliding} in Section \ref{ss:51} and the converse direction of Theorem \ref{thm1_colliding} in Section \ref{ss:52}.

\section{Dynamics and heights} \label{Sec2}

Let us recall the analysis of \cite{Ghi24, GH24}. While we use a similar analytical approach, our method to prove our primary results differs from theirs.

Let $L_0$ be the perfect closure of the rational function field $\olfp(t)$, i.e. $L_0:= \olfp(t,t^{1/p},t^{1/p^2}, \ldots )$ and $L$ be a finite extension of $L_0$. Fix an algebraic closure $\ol{L}$ of $L$. 
Then the set $\Omega:=\Omega_L$ of all (pairwise inequivalent) places of $L$ has the following properties:
\begin{enumerate}
\item[(i)] For every nonzero $x\in L$, we have $|x|_v=1$ for all but finitely many $v\in \Omega$.
\item[(ii)] For every nonzero $x\in L$, we have 
\begin{equation}\label{eq:prodfor}
\prod_{v\in \Omega} |x|_v=1,
\end{equation}
where $|\cdot|_v$ denotes the absolute value associated with $v\in \Omega$ and the absolute values $|\cdot|_v$ are normalized so that the product formula \eqref{eq:prodfor} holds.
\item[(iii)] For an element $x \in L$, we have $|x|_v \le 1$ for every $v \in \Omega$ if and only if $x \in \olfp$.
\end{enumerate}

\subsection{Canonical heights}\label{ss:21}
Define the Weil height for each $x\in \ol{L}$ as
$$
h(x):= \frac{1}{[L(x):L]} \sum_{v\in \Omega} \sum_{y \in \mr{Gal}(L^{\mr{sep}}/L)\cdot x}\log^+|y|_v,
$$
where $\log^+(z)=\log \max (z,1)$ for each real number $z>0$. The \emph{global canonical height} of $x\in \ol{L}$ with respect to the polynomial $f_\lambda \in \ol{L}[x]$ ($\lambda \in \ol{L}$) is defined by
$$
\widehat{h}_{f_{\lambda}}(x) :=\underset{n\rightarrow \infty}{\lim} \frac{h(f_\lambda^n(x))}{d^n},
$$
where $d$ is the degree of $f$ and $f^n :={\underbrace{f\circ \cdots \circ f}_{n\text{-times}}}$.

The global canonical height is directly related to the preperiodicity. An element $\gamma \in L$ is preperiodic under $f_{\lambda}$ if and only if $\widehat{h}_{f_\lambda}(\gamma)=0$ (see \cite[Remark 2.5]{GH24} for a detailed explanation).

For each $v \in \Omega$, let $\bC_v$ be the completion of an algebraic closure of the completion of $L$ at $v$. This is the smallest extension of $L_v$ that is both complete and algebraically closed. The absolute value $|\cdot |_v$ extends uniquely to $\bC_v$, and we use the same notation $|\cdot |_v$ for this extension.
Define the \emph{local canonical height} $\widehat{h}_{v,\lambda}(x)$ for $x\in \bC_v$ with respect to the polynomial $f_\lambda$ as follows:
$$
\widehat{h}_{v,\lambda}(x):= \underset{n\rightarrow \infty}{\lim}\frac{\log^+|f_\lambda^n(x)|_v}{d^n}.
$$
	
\subsubsection{Relation between heights} Let $\al, \be \in L$ and $f(x) \in \olfp[x]$ be a monic polynomial of degree at least $2$ satisfying $f(0)=0$. Let $S$ be the (finite) set of places $v \in \Omega$ such that
$$
\max(|\al|_v, |\be|_v) > 1.
$$
By the property (iii) of the set $\Omega$, $S$ is nonempty unless both $\al$ and $\be$ are contained in $\olfp$. The following theorem provides a relation between local and global canonical heights. 

\begin{thm} \label{thm:sameheight} 
(\cite[Theorem 2.13]{GH24}) Let $\al, \be \in L$. Assume that there exists an infinite sequence $(\lambda_n)_{n \ge 1}$ of distinct elements in $\ol{L}$ with the property that 
\begin{equation*}\label{eq:limitzero}
\underset{n \rightarrow \infty}{\lim} \widehat{h}_{f_{\lambda_n}}(\al)=\underset{n \rightarrow \infty}{\lim} \widehat{h}_{f_{\lambda_n}}(\be)=0.
\end{equation*}
Then for every $v \in \Omega$ and $\lambda\in \bC_v$, we have $\widehat{h}_{v,\lambda}(\al)=\widehat{h}_{v,\lambda}(\be)$.
\end{thm}

The assumption of Theorem \ref{thm:sameheight} is satisfied when $\lt \Prep(f; \al, \be) \rt = \infty$. Hence, if $\lt \Prep(f; \al, \be) \rt = \infty$, then $\widehat{h}_{v,\lambda}(\al)=\widehat{h}_{v,\lambda}(\be)$ for every $v \in \Omega$ and $\lambda \in \ol{L} \subseteq \bC_v$. As noted in the introduction, later in the paper we will choose $\lambda = \oal - f(\oal)$ for some $\oal \in \ol{L}$ such that $f(\oal)=f(\al)$.

\subsection{The additive case}\label{ss:add}
Let $f(x) \in \olfp[x]$ be an additive polynomial of degree $d \ge 2$. According to \cite[Remark 1.4]{GH24}, we have $\lt \Prep(f; \al, \be) \rt = \infty$ if and only if $\be-\al \in \olfp$. Since the proof of this fact is not given in \cite{GH24}, we illustrate the proof here. 

Suppose first that $\lt \Prep(f; \al, \be) \rt = \infty$. 
For $\lambda = \al - f(\al)$, $\al$ is preperiodic under $f_{\lambda}$. Then by Theorem \ref{thm:sameheight}, $\widehat{h}_{v,\lambda}(\al)= \widehat{h}_{v,\lambda}(\be)=0$ for every $v \in \Omega$. 
Let $\ep_0 := \be - \al$ and $\ep_{n+1} := f(\ep_n)$ for every $n \ge 0$. By the additivity of $f$, it follows that $f_{\lambda}^n(\be) = \al + \ep_n$ for every $n \ge 0$. 
If $|\ep_0|_v > 1$ for some $v \in \Omega$, then $|\ep_n|_v = |\ep_0|_v^{d^n}$ for every $n \ge 0$ so $|f_{\lambda}^n(\be)|_v = |\al + \ep_n|_v = |\ep_0|_v^{d^n}$ for all sufficiently large $n$. This contradicts the fact that $\widehat{h}_{v,\lambda}(\be)=0$. Thus $|\ep_0|_v \le 1$ for all $v \in \Omega$, which implies that $\ep_0 = \be - \al \in \olfp$.

Now suppose that $\be-\al \in \olfp$. Since there are infinitely many $\lambda \in \ol{L}$ such that $\al$ is preperiodic under $f_{\lambda}$ (\cite[Proposition 2.3]{GH24}), it is enough to show that if $\al$ is preperiodic under $f_\lambda$, then $\be$ is also preperiodic under $f_\lambda$. By induction on $n$, one can show that for each $n$ there is a polynomial $g_n \in \olfp[x]$ satisfying
$f_\lambda^n(x)=f^n(x)+g_n(\lambda)$. This implies that
\begin{equation} \label{eqn_add}
f_\lambda^n(\be)-f_\lambda^n(\al)= f^n(\be)-f^n(\al)=f^n(\be-\al).
\end{equation}

Assume that $\al$ is preperiodic under $f_\lambda$. Then there is a nonnegative integer $k$ and a positive integer $l$ such that
$f_{\lambda}^k(\al)=f_{\lambda}^{k+rl}(\al)$ for every $r \ge 0$. 
Choose a positive integer $m$ such that $\be - \al$ and the coefficients of $f$ are in $\bF_{p^m}$. 
Since $f^{k+rl}(\be-\al) \in \bF_{p^m}$ for every $r \ge 0$, there are integers $0 \le r_1 < r_2 \le p^m$ such that $f^{k+r_1l}(\be-\al)=f^{k+r_2l}(\be-\al)$. By \eqref{eqn_add}, we have $f_\lambda^{k+r_1l}(\be)-f_\lambda^{k+r_1l}(\al)=f_\lambda^{k+r_2l}(\be)-f_\lambda^{k+r_2l}(\al)$. Since $f_\lambda^{k+r_1l}(\al)=f_\lambda^{k+r_2l}(\al)$, we conclude that $f_\lambda^{k+r_1l}(\be)=f_\lambda^{k+r_2l}(\be)$ and hence $\be$ is preperiodic under $f_{\lambda}$.

\section{Bounds on the valuations} \label{Sec3}

Let $f = \sum_{i=1}^{r} c_ix^{d_i} \in \olfp[x]$ be a polynomial of degree $d \ge 2$ such that $c_i \in \olfp^*$ and $1 \le d_1 < \cdots < d_r = d$. Assume that the set $\Prep(f; \al, \be)$ is infinite, so 
$$
\widehat{h}_{v,\lambda}(\al)=\widehat{h}_{v,\lambda}(\be)
$$
for every $v \in \Omega$ and $\lambda \in \ol{L}$ by Theorem \ref{thm:sameheight}. 

Let $\ep_0 := \be - \al$, $\ep_{n+1} := f(\ep_n + \al) - f(\al)$ and $\ep_n' := \ep_n - \ep_0$ for each $n \ge 0$. 
Similarly, let $\de_0 := \al-\be$, $\de_{n+1} := f(\de_n + \be) - f(\be)$ and $\de_n' := \de_n - \de_0$ for each $n \ge 0$.

\subsection{Local absolute values of \texorpdfstring{$\ep_n$}{epn}} \label{ss:boundofvalue} 

Let $\lambda_{\al} := \al-f(\al)$. Then $\ep_n = f_{\lambda_{\al}}^n(\be)-\al$ for each $n \ge 0$, which can be proved by induction on $n$.
For every $v \in S$, $f_{\lambda_{\al}}(\al)=\al$ so $\widehat{h}_{v,\lambda}(\al)=0$. By the assumption that $\Prep(f; \al, \be)$ is infinite, we also have $\widehat{h}_{v,\lambda_{\al}}(\be)=0$. 

\begin{lem} \label{lem3a}
For every $v \in S$, we have $|\al|_v=|\be|_v$.
\end{lem}

\begin{proof}
Assume that $|\al|_v < |\be|_v$, so $|\be|_v>1$ by the definition of $S$. Then $|\ep_0|_v = |\be|_v > \max(|\al|_v, 1)$. Moreover, if $|\ep_n|_v > \max(|\al|_v, 1)$, then
$$
|\ep_{n+1}|_v
= \lt \sum_{i=1}^{r} c_i((\ep_n+\al)^{d_i}-\al^{d_i}) \rt_v
= |\ep_n|_v^{d}.
$$
Therefore $|\ep_n|_v=|\be|_v^{d^n}$ for every $n \ge 0$. Now we have $|f_{\lambda_{\al}}^n(\be)|_v = |\ep_n+\al|_v = |\be|_v^{d^n}$ so $\widehat{h}_{v,\lambda_{\al}}(\be)>0$, which is a contradiction.
\end{proof}

For each $v \in S$, denote $C_v := |\al|_v=|\be|_v > 1$. 

\begin{lem} \label{lem3b}
For every $v \in S$ and $n \ge 0$, we have $|\ep_n|_v, |\de_n|_v \le C_v$.
\end{lem}

\begin{proof}
Since $\ep_n$ and $\de_n$ are symmetric, it is enough to show that $|\ep_n|_v \le C_v$. Assume that $|\ep_n|_v>C_v$ for some $n$. Then for each $m \ge n$, we have $|\ep_m|_v = |\ep_n|_v^{d^{m-n}} > C_v^{d^{m-n}}$ so $|f_{\lambda_{\al}}^m(\be)|_v = |\ep_m+\al|_v > C_v^{d^{m-n}}$. This contradicts the fact that $\widehat{h}_{v,\lambda_{\al}}(\be)=0$.
\end{proof}

The preceding lemmas are analogous to the results of \cite[Section 3.1–3.2]{GH24}, but are established here using a different approach. The main difference is that only $|\al|_v$, $|\be|_v$ and $|\ep_1|_v$ were considered in \cite{GH24}, while our proof treats $|\ep_n|_v$ for arbitrary $n$. 

\begin{rmk} \label{rmk3c}
(Revisiting the result of Ghioca--Hsia) Let $f \in \olfp[x]$ be a polynomial which satisfies the assumptions in Theorem \ref{thm1b}. 
Ghioca and Hsia \cite[Lemma 3.9]{GH24} proved that if the set $\Prep(f; \al, \be)$ is infinite, then $|\ep_1|_v < 1$ for every $v \in S$. Here we provide an alternative proof of $|\ep_1|_v < 1$ using $\ep_2$. 

By \cite[Lemma 3.7]{GH24}, we have $|\ep_1|_v < C_v$ for every $v \in S$. Suppose $1 \le |\ep_1|_v < C_v$ and let
$$
u_i(k) := \frac{(k C_v^{s_r-1})^{p^{\ell_r}}}{(k C_v^{s_i-1})^{p^{\ell_i}}}
$$
for each $1 \le i \le r-1$. Then $u_i(1) = C_v^{p^{\ell_r}(s_r-1)-p^{\ell_i}(s_i-1)} > 1$ and $u_i(c_r)=C_v^{d_r-d_i}>1$ so $u_i(|\ep_1|_v)>1$. This implies that
$$
|(\ep_1 \al^{s_r-1})^{p^{\ell_r}}|_v 
\ge |(\ep_1 \al^{s_i-1})^{p^{\ell_i}}|_v 
\ge |(\ep_1^{j} \al^{s_i-j})^{p^{\ell_i}}|_v
$$ 
for every $1 \le i \le r$ and $1 \le j \le s_i$, and the equality $|(\ep_1 \al^{s_r-1})^{p^{\ell_r}}|_v = |(\ep_1^{j} \al^{s_i-j})^{p^{\ell_i}}|_v$ holds if and only if $i=r$ and $j=1$. Now we have
$$
|\ep_2|_v
=\lt \sum_{i=1}^{r} c_i \left ( (\ep_1+\al)^{s_i}-\al^{s_i} \right )^{p^{\ell_i}} \rt_v
= \lt \sum_{i=1}^{r} \sum_{j=1}^{s_i} c_i \binom{s_i}{j} (\ep_1^{j} \al^{s_i-j})^{p^{\ell_i}} \rt_v
= |(\ep_1 \al^{s_r-1})^{p^{\ell_r}}|_v > C_v, 
$$
which is impossible by Lemma \ref{lem3b}. 
\end{rmk}

In the remaining part of the paper, we assume that $f \in \olfp[x]$ is given as in Setup \ref{setup_f}, i.e.
$$
f(x) = c_1x^{d_1}+c_2x^{d_2} \in \olfp[x]
$$ 
where $c_1, c_2 \in \olfp^*$, $1 \le d_1 < d_2$ and $d_i = p^{\ell_i} s_i$ with $\ell_i \ge 0$ and $p \nmid s_i$. To simplify the proof, we will give some additional conditions on $f$, $\al$ and $\be$.
\begin{enumerate}
    \item[(i)] If $p^{\ell_2} (s_2-1) > \max(1, p^{\ell_1}(s_1-1))$, then $\lt \Prep(f; \al, \be) \rt = \infty$ if and only if $\al, \be \in \olfp$ or $f(\al)=f(\be)$ by Theorem \ref{thm1b}. Therefore we may assume that $p^{\ell_2}(s_2-1) \le \max(1, p^{\ell_1}(s_1-1))$.

    \item[(ii)] The case that $f$ is additive is covered in Section \ref{ss:add}, so we may assume that $f$ is not additive, i.e. $(s_1, s_2) \ne (1,1)$. If $s_1=1$, then the condition $p^{\ell_2} (s_2-1) \le 1$ implies that $\ell_2=0$ and $s_2=2$ so $d_2=2$ and $p>2$. However, this case was already covered by Theorem \ref{thm1a} (see \cite[Remark 1.3]{GH24}). Hence we may assume that $s_1>1$ and thus $p^{\ell_2} (s_2-1) \le p^{\ell_1}(s_1-1)$.

    \item[(iii)] We may assume that $S \ne \emptyset$ (so $\al, \be \notin \olfp$), as the case $\al, \be \in \olfp$ is trivial. 
    
    \item[(iv)] By Remark \ref{rmk_reduction}, we may assume that $f$ is monic, i.e. $c_2=1$.
\end{enumerate}

Now we are ready to introduce the following results on the local absolute values of $\ep_n$ and $\ep'_n$, which will be frequently used in Section \ref{Secmain}. 
By the conditions $d_1<d_2$ and $d_2-p^{\ell_2} \le d_1-p^{\ell_1}$, we have $\ell_2 > \ell_1$ and 
$$
\rho := \frac{d_2-d_1}{p^{\ell_2}-p^{\ell_1}} \in (0, 1].
$$
When $s_2=1$, we let 
$$
\rho' := \frac{1-\rho}{p^{\ell_1}}-s_1+1.
$$

\begin{lem} \label{lem3d}
For every $n \ge 0$, we have $\lt \ep_n \rt_v = C_v$ or $\lt \ep_n \rt_v = C_v^{1-\rho}$ or 
$$\lt \ep_n \rt_v = \lt \ep_{n+1} \rt_v^{1 / p^{\ell_1}} C_v^{-s_1+1} \le C_v^{1/p^{\ell_1}-s_1+1}.$$ 
When $s_2=1$, $\lt \ep_n \rt_v = C_v^{1-\rho}$ or $$\lt \ep_n \rt_v = \lt \ep_{n+1} \rt_v^{1 / p^{\ell_1}} C_v^{-s_1+1} \le C_v^{\rho'}.$$
\end{lem}
	
\begin{proof}
We have
\begin{equation} \label{eq3a}
\begin{split}
\ep_{n+1} & = f(\ep_n+\al)-f(\al) \\ 
& = c_1((\ep_n+\al)^{d_1}-\al^{d_1}) + ((\ep_n+\al)^{d_2}-\al^{d_2}) \\
& = c_1 \sum_{k=1}^{s_1} \binom{s_1}{k} \ep_n^{p^{\ell_1} k} \al^{p^{\ell_1}(s_1-k)} + \sum_{k=1}^{s_2} \binom{s_2}{k} \ep_n^{p^{\ell_2} k} \al^{p^{\ell_2}(s_2-k)}.
\end{split}    
\end{equation}
Consider the valuation $\lt \ep_n \rt_v$ for $v \in S$. By Lemma \ref{lem3b}, $\lt \ep_n \rt_v \le C_v$ for each $n$.

\begin{enumerate}
\item[(C1)] $\lt \ep_n \rt_v = C_v$.

\item[(C2)] If $1 \le \lt \ep_n \rt_v < C_v$, then
$$
\lt s_i \ep_n^{p^{\ell_i}} \al^{p^{\ell_i}(s_i-1)} \rt_v > \max_{2 \le k \le s_i} \lt \binom{s_i}{k} \ep_n^{p^{\ell_i} k} \al^{p^{\ell_i}(s_i-k)} \rt_v
$$
for $i=1, 2$ and
$$
\lt s_1 \ep_n^{p^{\ell_1}} \al^{p^{\ell_1}(s_1-1)} \rt_v = \lt \ep_n \rt_v^{p^{\ell_1}} C_v^{p^{\ell_1}(s_1-1)} \ge C_v.
$$
If $\lt \ep_n \rt_v>1$ or $p^{\ell_1}(s_1-1) > 1$, then $\lt s_1 \ep_n^{p^{\ell_1}} \al^{p^{\ell_1}(s_1-1)} \rt_v > C_v \ge \lt \ep_{n+1} \rt_v$ so 
$$
\lt s_1 \ep_n^{p^{\ell_1}} \al^{p^{\ell_1}(s_1-1)} \rt_v = \lt s_2 \ep_n^{p^{\ell_2}} \al^{p^{\ell_2}(s_2-1)} \rt_v.
$$
This implies that $\lt \ep_n \rt_v^{p^{\ell_1}} C_v^{p^{\ell_1}(s_1-1)} = \lt \ep_n \rt_v^{p^{\ell_2}} C_v^{p^{\ell_2}(s_2-1)}$ so 
$$ \lt \ep_n \rt_v = C_v^{\frac{p^{\ell_1}(s_1-1)-p^{\ell_2}(s_2-1)}{p^{\ell_2}-p^{\ell_1}}} = C_v^{1-\rho}. $$ 
If $\lt \ep_n \rt_v=1$ and $p^{\ell_1}(s_1-1)=1$, then $\ell_1=0$ and $s_1=2$ so 
$$
\lt \ep_{n+1} \rt_v = \lt c_1(\ep_n^2 + 2 \ep_n \al)+\ep_n^{p^{\ell_2}} \rt_v = C_v.
$$
In this case, the equality $\lt \ep_n \rt_v = \lt \ep_{n+1} \rt_v^{1/p^{\ell_1}} C_v^{-s_1+1}$ holds.
			
\item[(C3)] If $\lt \ep_n \rt_v < 1$, then $s_1 \ep_n^{p^{\ell_1}} \al^{p^{\ell_1}(s_1-1)}$ has the largest absolute value in \eqref{eq3a} so 
$$
\lt \ep_{n+1} \rt_v = \lt s_1 \ep_n^{p^{\ell_1}} \al^{p^{\ell_1}(s_1-1)} \rt_v = (\lt \ep_n \rt_v C_v^{s_1-1})^{p^{\ell_1}}
$$ 
and thus $\lt \ep_n \rt_v = \lt \ep_{n+1} \rt_v^{1 / p^{\ell_1}} C_v^{-s_1+1}$.
\end{enumerate}

It remains to show that if $s_2=1$, then $|\ep_n|_v < C_v$. If $s_2=1$ and $|\ep_n|_v = C_v$, then 
$$
\lt \sum_{k=1}^{s_1} \binom{s_1}{k} \ep_n^{p^{\ell_1} k} \al^{p^{\ell_1}(s_1-k)} \rt_v < \lt \sum_{k=1}^{s_2} \binom{s_2}{k} \ep_n^{p^{\ell_2} k} \al^{p^{\ell_2}(s_2-k)} \rt_v = C_v^{p^{\ell_2}}
$$
so $\lt \ep_{n+1} \rt_v = C_v^{p^{\ell_2}} >C_v$, which is a contradiction. 
\end{proof}

\begin{lem} \label{lem3e}
For every $n \ge 1$, we have $\lt \ep'_n \rt_v = C_v$ 
or $\lt \ep'_n \rt_v = C_v^{1-\rho}$ 
or $$\lt \ep'_n \rt_v = \lt \ep'_{n+1}-\ep'_1 \rt_v^{1 / p^{\ell_1}} C_v^{-s_1+1} \le C_v^{1/p^{\ell_1}-s_1+1}.$$ 
When $s_2=1$, $\lt \ep'_n \rt_v = C_v^{1-\rho}$ 
or $$\lt \ep'_n \rt_v = \lt \ep'_{n+1}-\ep'_1 \rt_v^{1 / p^{\ell_1}} C_v^{-s_1+1} \le C_v^{\rho'}.$$
\end{lem}
	
\begin{proof}
This can be done by applying the proof of Lemma \ref{lem3d} to the equation $\ep'_{n+1} - \ep'_1 = f(\ep'_n + \be) - f(\be)$.
\end{proof}

Replacing $\al$ with $\be$, we obtain the same results for $\lt \de_n \rt_v$ and $\lt \de'_n \rt_v$ for every $n \ge 1$. 

\begin{lem} \label{lem3:valep1}
We have $\lt \ep_1 \rt_v = C_v^{1-\rho}$ or $\lt \ep_1 \rt_v \le C_v^{1/p^{\ell_1}-s_1+1}$. 
\end{lem} 

\begin{proof}
By Lemma \ref{lem3d}, it suffices to show that $\lt \ep_1 \rt_v < C_v$. This follows from the proof of \cite[Lemma 3.7]{GH24}. We remark that the condition
$$
p^{\ell_r} (s_r-1) > \max ( 1, p^{\ell_1} (s_1-1), \ldots, p^{\ell_{r-1}} (s_{r-1}-1) )
$$
was not used in the proof of \cite[Lemma 3.7]{GH24}, so the same argument applies in our setting.
\end{proof}

\subsection{Valuations of elements in \texorpdfstring{$S(f,\al)$}{S(f,a)}} \label{ss:overlinealpha}

The goal of this section is to find an element $\oal_v$ in the set
$$
S(f,\al):= \{ \oal \in \ol{L} : f(\oal)=f(\al)\}
$$
for each $v \in S$ satisfying the condition $| \oal_v - \al |_v = C_v^{1-\rho}$.
For every $\oal \in S(f, \al)$ and $v \in S$, we have $|f(\oal)|_v=|f(\al)|_v=C_v^{d_2}$ so 
$|\oal|_v > 1$. Then $|f(\oal)|_v = |\oal|_v^{d_2} = C_v^{d_2}$ so we deduce that $|\oal|_v = C_v$.

\begin{lem}\label{lem3:oal}
For each $v \in S$, there is an element $\oal_v\in S(f,\al)$ such that $|\oal_v-\al|_v=C_v^{1-\rho}$.
\end{lem}

\begin{proof}
The polynomial $g(x) := f(x+\al)-f(\al) \in L[x]$ is given by
\begin{align*}
g(x) &= \left(c_1^{\frac{1}{p^{\ell_1}}} ( (x+\al)^{s_1}-\al^{s_1})+((x+\al)^{s_2}-\al^{s_2})^{p^{\ell_2-\ell_1}} \right)^{p^{\ell_1}}\\
&= \left ( x^{p^{\ell_2-\ell_1}s_2}+\cdots + c_1^{\frac{1}{p^{\ell_1}}} s_1\al^{s_1-1}x \right )^{p^{\ell_1}}\\
&= \left( x(x-x_1)\cdots (x-x_m)\right)^{p^{\ell_1}},
\end{align*}
where $m = p^{\ell_2-\ell_1}s_2 - 1$ and $x_1, \ldots, x_m \in \ol{L}$. (For $r \in \olfp$, $r^{\frac{1}{p^{\ell_1}}}$ is the unique element $r_1 \in \olfp$ such that $r_1^{p^{\ell_1}}=r$.) Since $c_1^{\frac{1}{p^{\ell_1}}} s_1 \al^{s_1-1} \ne 0$, the elements $x_1, \ldots, x_m$ are nonzero. 
Then we have 
$$
x_1 \cdots x_m = (-1)^{m} c_1^{\frac{1}{p^{\ell_1}}} s_1 \al^{s_1-1}
$$ 
by the relation between roots and coefficients. Since $m+1 = \frac{s_1d_2}{d_1} > s_1$, we have $| x_1 \cdots x_m |_v = C_v^{s_1-1} < C_v^m$. This implies that for every $v \in S$, there exists some $i$ (which depends on $v$) such that $|x_i|_v < C_v$. 

Now we let $\oal_v := x_i + \al$. Since $g(x_i)=f(x_i+\al)-f(\al)=0$, we have $\oal_v \in S(f, \al)$. To prove the lemma, we need to prove $|x_i|_v = C_v^{1-\rho}$. By the relation
$$
f(\oal_v) - f(\al) = c_1(\oal_v^{s_1}-\al^{s_1})^{p^{\ell_1}} + (\oal_v^{s_2}-\al^{s_2})^{p^{\ell_2}} = 0,
$$
we have
$$
\lt \oal_v^{s_1} - \al^{s_1} \rt_v^{p^{\ell_1}}
= \lt \oal_v^{s_2} - \al^{s_2} \rt_v^{p^{\ell_2}}.
$$
If $s$ is a positive integer such that $p \nmid s$, then
\begin{equation*} \label{eq:val-pnmids}
\begin{split}
\lt \oal_v^{s} - \al^{s} \rt_v
& = \lt \oal_v - \al \rt_v \lt \sum_{i=0}^{s-1} \oal_v^{s-1-i} \al^{i} \rt_v \\
& = \lt \oal_v - \al \rt_v \lt (\oal_v-\al) \sum_{i=1}^{s-1} i \oal_v^{s-1-i}\al^{i-1} + s \al^{s-1} \rt_v \\
& = \lt \oal_v - \al \rt_v C_v^{s-1}.
\end{split}
\end{equation*}
(The last equality follows from the condition $\lt \oal_v - \al \rt_v = \lt x_i \rt_v < C_v$.) We conclude that
$$
(\lt \oal_v - \al \rt_v C_v^{s_1-1})^{p^{\ell_1}}
= (\lt \oal_v - \al \rt_v C_v^{s_2-1})^{p^{\ell_2}},
$$
which implies that $\lt \oal_v - \al \rt_v = C_v^{1-\rho}$.
\end{proof}

Regarding the proof of Lemma \ref{lem3:oal}, we have the following remarks. These remarks will be used in the proof of Theorem \ref{thm4d}.

\begin{rmk} \label{rmk3oal}
For every $\oal \in S(f, \al) \setminus \{ \al \}$ and $v \in S$, either $|\oal-\al|_v =C_v^{1-\rho}$ or $|\oal-\al|_v =C_v$. If $s_2=1$, then 
$$
\lt x_1 \cdots x_m \rt_v = C_v^{s_1-1} = C_v^{m(1-\rho)}
$$
so $|\oal-\al|_v =C_v^{1-\rho}$ for every $\oal \in S(f, \al) \setminus \{ \al \}$ and $v \in S$.

Now choose any $\oal_1 \in S(f,\al)$. Then $\al$ is preperiodic under $f_{\lambda}$ if and only if $\oal_1$ is preperiodic under $f_{\lambda}$. Replacing $\al$ with $\oal_1$, one can show that for every $v \in S$ and $\oal_1, \oal_2 \in S(f, \al)$ with $\oal_1 \ne \oal_2$, we have $\lt \oal_1-\oal_2 \rt_v = C_v^{1-\rho}$ or $\lt \oal_1-\oal_2 \rt_v = C_v$. Moreover, if $s_2=1$, then $\lt \oal_1-\oal_2 \rt_v = C_v^{1-\rho}$.
\end{rmk}

\begin{rmk} \label{rmk3oal2}
Let $g_1(x) := c_1^{\frac{1}{p^{\ell_1}}}x^{s_1} + x^{p^{\ell_2-\ell_1}s_2} \in L[x]$ (so $f(x)=g_1(x)^{p^{\ell_1}}$) and $g_2(x) := g_1(x)-g_1(\al) \in L[x]$. By the assumption $\al \notin \olfp$, we have $g_2(0) = -g_1(\al) \ne 0$. This implies that $g_2(x)$ and $g_2'(x) = c_1^{\frac{1}{p^{\ell_1}}}s_1 x^{s_1-1}$ have no common root, so $g_2(x)$ is separable. Thus we have
$$
\lt S(f, \al) \rt = \lt \{ \oal \in \ol{L} : g_2(\oal)=0 \} \rt 
= \deg(g_2) = p^{\ell_2-\ell_1}s_2.
$$
\end{rmk}

Let $\oal \in S(f,\al)$ and $\lamoa := \oal - f(\oal)$. Then $f_{\lamoa}(\al)=f_{\lamoa}^2(\al)=\oal$ so $\al$ is preperiodic under $f_{\lamoa}$ and thus $\widehat{h}_{v, \lamoa}(\al)=\widehat{h}_{v, \lamoa}(\be)=0$. 
Let $\ol{\ep_0} := \be - \oal$, $\ol{\ep_{n+1}} := f(\ol{\ep_n}+\oal)-f(\oal)$ and $\ol{\ep_n}' := \ol{\ep_n}-\ol{\ep_0}$ for every $n \ge 0$. 
Similarly, let $\ol{\de_0} := \oal-\be$, $\ol{\de_{n+1}} := f(\ol{\de_n} + \be) - f(\be)$ and $\ol{\de_n}' := \ol{\de_n} - \ol{\de_0}$ for each $n \ge 0$. 
(The definitions of $\ol{\ep_n}$, $\ol{\ep_n}'$, $\ol{\de_n}$ and $\ol{\de_n}'$ depend on the choice of $\oal \in S(f,\al)$.)

\begin{lem} \label{lem3:valoal}
For every $n \ge 1$ and $\oal \in S(f,\al)$, we have $\lt \ol{\ep_n} \rt_v = C_v$ 
or $\lt \ol{\ep_n} \rt_v = C_v^{1-\rho}$ 
or $$\lt \ol{\ep_n} \rt_v = \lt \ol{\ep_{n+1}} \rt_v^{1 / p^{\ell_1}} C_v^{-s_1+1} \le C_v^{1/p{^{\ell_1}}-s_1+1}.$$ 
When $s_2=1$, we get $\lt \ol{\ep_n} \rt_v = C_v^{1-\rho}$ 
or $$\lt \ol{\ep_n} \rt_v = \lt \ol{\ep_{n+1}} \rt_v^{1 / p^{\ell_1}} C_v^{-s_1+1} \le C_v^{\rho'}.$$
\end{lem}

\begin{proof}
Assume that $\lt \ol{\ep_n} \rt_v > C_v$ for some $v \in S$ and $n \ge 0$. By the relation $\ol{\ep_{n+1}} = f(\ol{\ep_n}+\oal)-f(\oal)$, we have $\lt \ol{\ep_m} \rt_v = \lt \ol{\ep_n} \rt_v^{d_2^{m-n}} > C_v^{d_2^{m-n}}$ for every $m \ge n$ so $\lt f^{m}_{\lamoa}(\be) \rt_v = \lt \ol{\ep_m}+\oal \rt_v > C_v^{d_2^{m-n}}$. 
This contradicts the fact that $\widehat{h}_{v, \lamoa}(\be)=0$, so $\lt \ol{\ep_n} \rt_v \le C_v$ for every $v \in S$ and $n \ge 0$. Using the relation $\ol{\ep_{n+1}} = f(\ol{\ep_n}+\oal)-f(\oal)$, we can finish the proof by imitating the proof of Lemma \ref{lem3d}.
\end{proof}

\begin{lem} \label{lem3:valoal'}
For every $n \ge 1$ and $\oal \in S(f,\al)$, we have $\lt \ol{\ep_n}' \rt_v = C_v$ 
or $\lt \ol{\ep_n}' \rt_v = C_v^{1-\rho}$ 
or $$\lt \ol{\ep_n}' \rt_v = \lt \ol{\ep_{n+1}}'-\ol{\ep_1}' \rt_v^{1 / p^{\ell_1}} C_v^{-s_1+1} \le C_v^{\frac{1}{p^{\ell_1}}-s_1+1}.$$
When $s_2=1$, $\lt \ol{\ep_n}' \rt_v = C_v^{1-\rho}$ 
or $$\lt \ol{\ep_n}' \rt_v = \lt \ol{\ep_{n+1}}'-\ol{\ep_1}' \rt_v^{1 / p^{\ell_1}} C_v^{-s_1+1} \le C_v^{\rho'}.$$
\end{lem}

\begin{proof}
Apply the proof of Lemma \ref{lem3d} to the equation $\ol{\ep_{n+1}}'-\ol{\ep_1}'=f(\ol{\ep_n}'+\be)-f(\be)$. (We remark that $\lt \ol{\ep_n}' \rt_v = \lt \ol{\ep_n} - \ol{\ep_0} \rt_v \le C_v$ by Lemma \ref{lem3:valoal}.)
\end{proof}

Replacing $\oal$ with $\be$, we obtain the same results for $| \ol{\de_n} |_v$ and $| \ol{\de_n}' |_v$ for every $n \ge 1$.

\section{Proof of the main theorems} \label{Secmain}

Assume that $f \in \olfp[x]$ is given as in Setup \ref{setup_f}.
For each $v \in S$, let $S_v(f,\al)$ be the subset of $S(f,\al)$ consisting of $\oal \in S(f,\al)$ such that $|\oal-\al|_v=C_v^{1-\rho}$. By Lemma \ref{lem3:oal}, the set $S_v(f,\al)$ is nonempty for each $v \in S$. For a given $\oal \in S(f,\al)$, the element $\ol{\ep_2} - \ep_2$ (depends on $\oal$) is given by
\begin{equation} \label{eq:thm42-a}
\begin{split}
\ol{\ep_2}-\ep_2 
&= (f(\ep_1+\oal)-f(\oal))-(f(\ep_1+\al)-f(\al)) \\
&= c_1((\ep_1+\oal)^{s_1} - \oal^{s_1} - (\ep_1+\al)^{s_1} + \al^{s_1})^{p^{\ell_1}} + ((\ep_1+\oal)^{s_2} - \oal^{s_2} - (\ep_1+\al)^{s_2} + \al^{s_2})^{p^{\ell_2}} \\
&= c_1 \Big( \sum_{j=1}^{s_1-1} \binom{s_1}{j} \ep_1^{s_1-j}(\oal^{j}-\al^j) \Big)^{p^{\ell_1}} 
+ \Big( \sum_{j=1}^{s_2-1} \binom{s_2}{j} \ep_1^{s_2-j}(\oal^{j}-\al^j) \Big)^{p^{\ell_2}}.
\end{split}
\end{equation}

\subsection{The case \texorpdfstring{$s_2=1$}{s2=1}} \label{ss:41} 
In this section, we prove Theorem \ref{thm_main4} under the assumption $s_2=1$. In this case, the condition in Theorem \ref{thm_main4} is satisfied automatically. By some technical reasons, we prove the case $s_1=2$ (Theorem \ref{thm4c}) first and then prove the case $s_1 \ge 3$ (Theorem \ref{thm4d}). 

\begin{thm} \label{thm4c} 
Under Setup \ref{setup_f}, assume that $(s_1, s_2)=(2,1)$. Then $\lt \Prep(f; \al, \be) \rt = \infty$ if and only if $\al,\be \in \olfp$ or $f(\al)=f(\be)$.
\end{thm}
	
\begin{proof}
Assume that $\lt \Prep(f; \al, \be) \rt = \infty$, $S \ne \emptyset$ and $\ep_1\ep'_1 \ne 0$. By Lemma \ref{lem3d} (and its analogue for $\de_n$), the condition $s_2=1$ implies that $\lt \ep_2+\de_2 \rt_v \le C_v^{1-\rho}$.
Since
$$
\ep_2+\de_2
= f(\ep_1+\al)-f(\al)+f(-\ep_1+\be)-f(\be)
= c_1(2 \ep_1 \ep'_1)^{p^{\ell_1}},
$$
we have 
$$
\lt \ep_1 \ep'_1 \rt_v = \lt \ep_2+\de_2 \rt_v^{1/p^{\ell_1}} \le C_v^{1-\rho} < (C_v^{1-\rho})^2. 
$$
Hence at least one of $\lt \ep_1 \rt_v$ and $\lt \ep'_1 \rt_v$ is smaller than $C_v^{1-\rho}$. Lemmas \ref{lem3d} and \ref{lem3e} imply that
$$
\lt \ep_1\ep'_1 \rt_v \le C_v^{1-\rho+\rho'}
= C_v^{(s_1-1) \frac{2p^{\ell_1}-p^{\ell_2}+1}{p^{\ell_2}-p^{\ell_1}}}
= C_v^{\frac{d_1-d_2+1}{p^{\ell_2}-p^{\ell_1}}} \le 1.
$$
(Recall that we write $\rho' = \frac{1-\rho}{p^{\ell_1}}-s_1+1$ when $s_2=1$.) By the product formula, we have $\lt \ep_1\ep'_1 \rt_v = 1$ for every $v \in S$. This only happens when one of $\lt \ep_1 \rt_v$ and $\lt \ep'_1 \rt_v$ is $C_v^{1-\rho}>1$ and the other is $C_v^{\rho-1} < 1$. 
In this case, $\lt \ep_0 \rt_v = \lt \ep_1 - \ep'_1 \rt_v = C_v^{1-\rho}$ so
$$
\lt \frac{\ep_1 \ep'_1}{\ep_0} \rt_v \le \frac{1}{C_v^{1-\rho}} < 1
$$
for every $v \in S$. We also have $\ep_1 = f(\be)-f(\al) = \ep_0 g(\al, \be)$ for some $g(x,y) \in \olfp[x,y]$ so 
$$
\lt \frac{\ep_1 \ep'_1}{\ep_0} \rt_v
= \lt g(\al, \be) \ep'_1 \rt_v \le 1
$$
for every $v \notin S$. This contradicts the product formula, so we conclude that $\ep_1 \ep'_1=0$. 

Choose any $\oal \in S(f,\al) \setminus \{ \al \}$ and apply the same procedure to 
$$
\ol{\ep_2}+\ol{\de_2} = (2 \ol{\ep_1} \ol{\ep_1}')^{p^{\ell_1}} = (2 \ep_1 \ol{\ep_1}')^{p^{\ell_1}}.
$$
Then we deduce that $\ep_1 \ol{\ep_1}'=0$. Since $\ol{\ep_1}' = \ep_1 - \ol{\ep_0} = \ep'_1 + (\ol{\al}-\al) \ne \ep'_1$, at least one of $\ep'_1$ and $\ol{\ep_1}'$ is nonzero. Thus we have $\ep_1=0$. 

Now we prove the ``if'' part of the proof. If $\al, \be \in \olfp$, then both $\al$ and $\be$ are preperiodic under $f_\lambda$ for each $\lambda \in \olfp$ so the set $\Prep(f; \al, \be)$ is infinite. If $\ep_1=0$, then $f_\lambda(\al)=f_\lambda(\be)$ for every $\lambda \in \ol{L}$ so $\al$ is preperiodic under $f_\lambda$ if and only if $\be$ is preperiodic under $f_\lambda$.
Since there are infinitely many $\lambda \in \ol{L}$ such that $\al$ is preperiodic under $f_\lambda$ (\cite[Proposition 2.3]{GH24}), we conclude that the set $\Prep(f; \al, \be)$ is infinite.
\end{proof}

\begin{thm}\label{thm4d}
Under Setup \ref{setup_f}, assume that $s_1 \ge 3$ and $s_2=1$. Then $\lt \Prep(f; \al, \be) \rt = \infty$ if and only if $\al,\be \in \olfp$ or $f(\al)=f(\be)$.
\end{thm}

\begin{proof}
Assume that $\lt \Prep(f; \al, \be) \rt = \infty$ and $S \ne \emptyset$. For every $v \in S$ and $\oal \in S(f,\al) \setminus \{ \al \}$, $\lt \oal-\al \rt_v \ge C_v^{1-\rho}$ by Remark \ref{rmk3oal} and $|\ol{\ep_2}-\ep_2|_v \le C_v^{1-\rho}$ by Lemmas \ref{lem3d} and \ref{lem3:valoal}.
By the assumption $s_2=1$, \eqref{eq:thm42-a} gives
$$
\lt \frac{(\ol{\ep_2}-\ep_2)^{1/p^{\ell_1}}}{\oal - \al} \rt_v
= \lt \sum_{j=1}^{s_1-1} \binom{s_1}{j}\ep_1^{s_1-j}\frac{\oal^j-\al^j}{\oal-\al} \rt_v
= \lt g(\oal) \rt_v
\le 1
$$
for every $v \in S$ and $\oal \in S(f,\al) \setminus \{ \al \}$,
where the polynomial $g \in L[x]$ is given by
$$
g(x) := \sum_{j=1}^{s_1-1} \binom{s_1}{j}\ep_1^{s_1-j}\frac{x^j-\al^j}{x-\al} \in L[x].
$$

Choose $\oal_1, \oal_2 \in S(f,\al) \setminus \{ \al \}$ with $\oal_1 \ne \oal_2$. Then $\lt \oal_2-\oal_1 \rt_v = C_v^{1-\rho}$ by Remark \ref{rmk3oal} so
$$
\lt \frac{g(\oal_2)-g(\oal_1)}{\oal_2 - \oal_1} \rt_v < 1
$$
for every $v \in S$. We also have 
$$
\lt \frac{g(\oal_2)-g(\oal_1)}{\oal_2 - \oal_1} \rt_v \le 1
$$
for every $v \notin S$ so the product formula implies that $g(\oal_1)=g(\oal_2)$. 

Fix an element $\oal_0 \in S(f,\al) \setminus \{ \al \}$. By the above argument, we have $g(\oal)=g(\oal_0)$ for every $\oal \in S(f,\al) \setminus \{ \al \}$. By Remark \ref{rmk3oal2},
$$
\lt S(f,\al) \setminus \{ \al \} \rt = p^{\ell_2-\ell_1}s_2-1
$$
so the equation $g(x)=g(\oal_0)$ has $p^{\ell_2-\ell_1}s_2-1$ distinct roots. 

Assume that $\ep_1 \ne 0$, so the polynomial $g$ has degree $s_1-2 \ge 1$. Fix an element $\oal_0 \in S(f,\al) \setminus \{ \al \}$. Then the equation $g(x)=g(\oal_0)$ has at most $s_1-2$ solutions in $\ol{L}$. 
However, we have shown that $g(\oal)=g(\oal_0)$ for every $\oal \in S(f,\al) \setminus \{ \al \}$ and 
$$
\lt S(f,\al) \setminus \{ \al \} \rt = p^{\ell_2-\ell_1}s_2-1 > s_1-2
$$
by Remark \ref{rmk3oal2}. This contradiction implies that $\ep_1=0$. The proof of the ``if'' part can be done as in the proof of Theorem \ref{thm4c}.
\end{proof}

\subsection{The case \texorpdfstring{$s_2>1$}{s2>1} and \texorpdfstring{$\rho<1$}{rho>1}} \label{ss:42} 

In this section, we prove Theorem \ref{thm_main4} under the assumption $s_2>1$. The following theorem is the most important result in this paper, and its proof is also the most technical.

\begin{thm} \label{thm4_general}
Under Setup \ref{setup_f}, assume that $s_2>1$ and
$$
p^{\ell_2}(s_2-1) < p^{\ell_1}(s_1-1).
$$
Then $\lt \Prep(f; \al, \be) \rt = \infty$ if and only if $\al,\be\in \olfp$ or $f(\al)=f(\be)$.
\end{thm}

\begin{proof}
Assume that $\lt \Prep(f; \al, \be) \rt = \infty$ and $S \ne \emptyset$.
Let $m := p^{\ell_2-\ell_1}s_2-1$ and $r := p^{\ell_2-\ell_1}$. By the assumption $p^{\ell_2}(s_2-1) < p^{\ell_1}(s_1-1)$, we have $\rho<1$ so
$$
\frac{1}{1-\rho} = \frac{p^{\ell_2-\ell_1}-1}{(s_1-1)-(s_2-1)p^{\ell_2-\ell_1}} \le p^{\ell_2-\ell_1}-1 < r.
$$
By Remark \ref{rmk3oal2}, $\lt S(f,\al) \rt = p^{\ell_2-\ell_1}s_2 = m+1$ so $S(f,\al)=\{ \al, \oal_1, \ldots, \oal_m \}$ for some $\oal_1, \ldots, \oal_m \in \ol{L}$. Define the polynomials $g_1, \ldots, g_r \in \ol{L}[x]$ as follow:
\begin{align*}
g_1(x) & = \frac{(f(\ep_1+x)-f(x)-(f(\ep_1+\al)-f(\al)))^{1/p^{\ell_1}}}{x-\al}, \\
g_{i+1}(x) & = \frac{g_i(x)-g_i(\oal_i)}{x-\oal_i} \;\; (1 \le i \le r-1).
\end{align*}
Since $\lt \oal_i-\al \rt_v \ge C_v^{1-\rho}$ by Remark \ref{rmk3oal}, we have
$$
\lt g_1(\oal_i) \rt_v = \lt \frac{(\ol{\ep_2}-\ep_2)^{1/p^{\ell_1}}}{\oal_i-\al} \rt_v \le C_v^{1-(1-\rho)}
$$
for every $1 \le i \le m$. (Here, $\ol{\ep_2}$ depends on the choice of $\oal_i$.) By Remark \ref{rmk3oal}, $\lt \oal_i-\oal_1 \rt_v \ge C_v^{1-\rho}$ so
$$
\lt g_2(\oal_i) \rt_v = \lt \frac{g_1(\oal_i)-g_1(\oal_1)}{\oal_i-\oal_1} \rt_v \le C_v^{1-2(1-\rho)}
$$
for every $2 \le i \le m$. Iterating this process, we deduce that
$$
\lt g_r(\oal_i) \rt_v \le C_v^{1-r(1-\rho)} < 1
$$
for every $v \in S$ and $r \le i \le m$. By the product formula, $g_r(\oal_i)=0$ for every $r \le i \le m$.

Now assume that $\ep_1 \ne 0$. Following \eqref{eq:thm42-a}, we have
\begin{align*}
(x-\al)g_1(x) & = c_1^{1/p^{\ell_1}} \sum_{j=1}^{s_1-1} \binom{s_1}{j} \ep_1^{s_1-j}(x^{j}-\al^j) + \Big( \sum_{j=1}^{s_2-1} \binom{s_2}{j} \ep_1^{s_2-j}(x^{j}-\al^j) \Big)^{p^{\ell_2-\ell_1}} \\
& = c_1^{1/p^{\ell_1}}s_1 \ep_1 x^{s_1-1} + (\text{lower terms})
\end{align*}
and $c_1^{1/p^{\ell_1}}s_1 \ep_1 \ne 0$ so $\deg(g_1)=s_1-2$. 
By the assumptions $\rho<1$ and $s_2 > 1$, we have
$$
(s_1-1)-r > (s_2-1)p^{\ell_2-\ell_1} - r = (s_2-2)p^{\ell_2-\ell_1} \ge 0
$$
so $\deg(g_r) = (s_1-1)-r > 0$. However, the polynomial $g_r$ has $m-r+1$ distinct roots $\oal_r, \oal_{r+1}, \ldots, \oal_{m} \in \ol{L}$ and $m-r+1=p^{\ell_2-\ell_1}s_2-r > (s_1-1)-r$, which is a contradiction. We conclude that $\ep_1=0$. The proof of the ``if'' part can be done as in the proof of Theorem \ref{thm4c}.
\end{proof}

During the proof of Theorem \ref{thm4_general}, the assumption $s_2>1$ was needed to guarantee $\deg(g_r)>0$. This is the reason that the case $s_2=1$ is treated separately.

\subsection{The case \texorpdfstring{$\rho=1$}{rho=1}} \label{ss:43}

In this section, we consider the case $\rho=1$. If $s_1 = s_2 = 1$, then $f$ is additive, and this case was covered in Section \ref{ss:add}. We therefore assume $s_2 > 1$ (and hence $s_1 > 1$).

\begin{thm} \label{thm4_rho1}
Under Setup \ref{setup_f}, assume that $s_1, s_2>1$ and $\rho=1$. If $\lt \Prep(f; \al, \be) \rt = \infty$ and not both $\al$ and $\be$ are contained in $\olfp$, then either $\ep_1=0$ or
$$
\ep_1^{p^{\ell_2}-p^{\ell_1}} = -\frac{c_1s_1}{c_2s_2}.
$$
\end{thm}

\begin{proof}
Assume that $\lt \Prep(f; \al, \be) \rt = \infty$ and $S \ne \emptyset$.
By Lemma \ref{lem3:valep1}, $|\ep_1|_v \le C_v^{1-\rho}=1$ for every $v \in S$ so $\ep_1 \in \olfp$. For $m = p^{\ell_1}(s_1-1)=p^{\ell_2}(s_2-1) > 1$,
\begin{align*}
\ep_2 & = c_1((\ep_1+\al)^{s_1}-\al^{s_1})^{p^{\ell_1}} + c_2((\ep_1+\al)^{s_2}-\al^{s_2})^{p^{\ell_2}} \\
& = (c_1s_1\ep_1^{p^{\ell_1}}+c_2s_2\ep_1^{p^{\ell_2}})\al^m + (\text{lower terms}).
\end{align*}
If $c_1s_1\ep_1^{p^{\ell_1}}+c_2s_2\ep_1^{p^{\ell_2}} \ne 0$, then $\lt \ep_2 \rt_v = C_v^m > C_v$, which is impossible by Lemma \ref{lem3b}. Hence 
$$
c_1s_1\ep_1^{p^{\ell_1}}+c_2s_2\ep_1^{p^{\ell_2}} = 0
$$
so either $\ep_1=0$ or $\ep_1^{p^{\ell_2}-p^{\ell_1}} = -\frac{c_1s_1}{c_2s_2}$.
\end{proof}

\begin{rmk} \label{rmk4_limit}
As a final remark, we present an example illustrating the limitation of our method. Let $p=3$ and $f(x)=x^4+x^6 \in \olfp[x]$. In this case, $s_2>1$ and $\rho=1=\frac{6-4}{3-1}=1$ so Theorem \ref{thm4_rho1} implies that either $\ep_1=0$ or $\ep_1^{3-1} = -\frac{4}{2} = 1 \in \ol{\bF}_3$. Hence $\ep_1 \in \{ -1, 0, 1\}$. 

However, our method does not allow us to rule out the cases $\ep_1=1$ and $\ep_1=-1$. If $\ep_1 = \pm 1$, then a direct computation shows that for every $\oal \in S(f,\al)$, we have $\ol{\ep_n} = \pm \oal - 1$ for all $n \ge 2$. In this case, the absolute values $\lt \ol{\ep_n} \rt_v$ ($n \ge 2$, $v \in S$) are always $C_v$, which provide no new information. This shows that the methods developed in this paper are not sufficient to establish $\ep_1=0$ when $\rho=1$.
\end{rmk}

\section{Colliding orbits problem} \label{Sec_col}

In this section, we prove Theorem \ref{thm1_colliding} on the colliding orbits problem. Throughout this section, we denote $C_f(\al_1,\al_2;\be)$ by $C(\al_1,\al_2;\be)$ for simplicity. 
Since our proof is based on the proof of \cite[Theorem 1.1]{AG25} in the case where $d$ is not a power of $p$, we briefly summarize it for the convenience of the reader.

For the forward direction, the argument in \cite{AG25} proceeds as follows. If $C(\al_1,\al_2;\be)$ is infinite, then there is an infinite sequence $(\lambda_k)_{k \ge 1}$ of distinct elements in $\ol{L}$ with the property that
$$
\lim_{k \to \infty} \widehat{h}_{f_{\lambda_k}}(\al_1) = \lim_{k \to \infty} \widehat{h}_{f_{\lambda_k}}(\al_2) = 0
$$
by \cite[Proposition 6.2]{AG25}. Theorem \ref{thm:sameheight} implies that  
$$
\widehat{h}_{v, \lambda}(\al_1) = \widehat{h}_{v, \lambda}(\al_2)
$$
for every $\lambda \in \ol{L}$ and $v \in \Omega_L$.
By \cite[Proposition 5.1]{Ghi24}, the equality of local canonical heights implies $\al_1^d = \al_2^d$. The converse direction follows from \cite[Theorem 2.4]{AG25}, which states that for every $\al, \be \in L$, there are infinitely many $\lambda \in \ol{L}$ such that $f_\lambda^m(\al) = \be$ for some $m \in \bN$.

\subsection{Forward direction of Theorem \ref{thm1_colliding}} \label{ss:51}

In this section, we prove the forward direction of Theorem \ref{thm1_colliding}.  
By Remark \ref{rmk:colliding-monic}, we may assume that $f$ is monic, so $f(x)=x^{d_2}+c_1x^{d_1}$. 
We begin with the following lemma, which describes the iterates $f_\lambda^n(x)$ as a polynomial in $\lambda$. It is an analogue of \cite[Lemmas 3.2 and 3.8]{AG25}, but the proof is more complicated since the polynomial $f$ is not a monomial in our setting.

\begin{lem}\label{lem5:polynomialproperty}
For each $n \in \bN$, 
\begin{equation} \label{eq5_iteration}
f_\lambda^n(x)=\sum_{i=0}^{d_2^{n-1}} c_{n,i}(x) \lambda^{d_2^{n-1}-i}
\end{equation}
for some polynomials $c_{n,i} \in \olfp[x]$. Furthermore, the polynomials $c_{n,i}(x)$ satisfy the following:
\begin{enumerate}
    \item[(a)] $c_{n,0}(x)=1$ and $c_{n,d_2^{n-1}}(x) = f^n(x)$;
    
    \item[(b)] $\deg(c_{n,i})\leq d_2\cdot i$ for $i=0,1,\ldots , d_2^{n-1}$;
    
    \item[(c)] Among the terms of $f_{\lambda}^n(x)$ whose coefficients $c_{n,i}(x)$ have positive degree in $x$, the term with the highest degree in $\lambda$ is $a_n(x) \lambda^{b_n}$ where
    $$
    a_1(x)=f(x), \; b_1=0, \; a_2(x)=s_1f(x)^{p^{\ell_1}}, \; b_2=p^{\ell_1}(s_1-1)
    $$
    and
    \begin{align*}
    b_{n+1} & = p^{\ell_2}(d_2^{n-1}(s_2-1)+b_n), \\
    a_{n+1}(x) & = s_2a_n(x)^{p^{\ell_2}}+u_n
    \end{align*}
    for some $u_n \in \olfp$ for every $n \ge 2$.
\end{enumerate}
\end{lem}

\begin{proof}
By induction on $n$, one can easily prove that \eqref{eq5_iteration} holds for every $n \in \bN$. When $n=1$, $f_{\lambda}(x)=f(x)+\lambda$ so $c_{n,0}(x)=1$. If $c_{n,0}(x)=1$ for a given $n \ge 1$, then
$$
f_\lambda^{n+1}(x)
= \left ( \lambda^{d_2^{n-1}} + \sum_{i=1}^{d_2^{n-1}} c_{n,i}(x) \lambda^{d_2^{n-1}-i} \right )^{d_2}+c_1\left ( \lambda^{d_2^{n-1}} + \sum_{i=1}^{d_2^{n-1}} c_{n,i}(x) \lambda^{d_2^{n-1}-i} \right )^{d_2}+\lambda
$$
so $c_{n+1, 0}(x) = 1$. By \eqref{eq5_iteration}, $c_{n,d_2^{n-1}}(x) = f_0^n(x) = f^n(x)$ so the assertion (a) holds. 

The proof of (b) also can be done by induction on $n$. The case $n=1$ is trivial since $c_{1,0}(x)=1$ and $c_{1,1}(x)=f(x)$. Assume that the assertion (b) holds for a given $n \ge 1$. Then
$$
f_\lambda^{n+1}(x)
= \left ( \sum_{i=0}^{d_2^{n-1}} c_{n,i}(x) \lambda^{d_2^{n-1}-i} \right )^{d_2}+c_1\left ( \sum_{i=0}^{d_2^{n-1}} c_{n,i}(x) \lambda^{d_2^{n-1}-i} \right )^{d_1}+\lambda
$$
so 
$$
c_{n+1, i}(x) = \sum_{j_1 + \cdots + j_{d_2} = i} c_{n,j_1}(x) \cdots c_{n,j_{d_2}}(x)
+ c_1 \sum_{j_1 + \cdots + j_{d_2} = i-d_2^{n-1}(d_2-d_1)} c_{n,j_1}(x) \cdots c_{n,j_{d_2}}(x)
+ \ep_{n,i},
$$
where $\ep_{n,i}=1$ if $i=d_2^n-1$ and $\ep_{n,i}=0$ otherwise. Then
$$
\deg (c_{n,j_1}(x) \cdots c_{n,j_{d_2}}(x)) \le d_2 \cdot (j_1 + \cdots +j_{d_2}) \le d_2 \cdot i
$$
whenever $j_1 + \cdots + j_{d_2} = i$ or $j_1 + \cdots + j_{d_2} = i-d_2^{n-1}(d_2-d_1)$ so $\deg c_{n+1,i}(x) \le d_2 \cdot i$.

Now we prove the assertion (c). The case $n=1$ is trivial. When $n=2$, we have
\begin{align*}
f_\lambda^2(x) 
& = (f(x)^{p^{\ell_2}}+\lambda^{p^{\ell_2}})^{s_2} +c_1 (f(x)^{p^{\ell_1}}+\lambda^{p^{\ell_1}})^{s_1} + \lambda \\
& = \sum_{j=0}^{s_2} \binom{s_2}{j} f(x)^{p^{\ell_2}j} \lambda^{p^{\ell_2} (s_2-j)} + \sum_{j=0}^{s_1} \binom{s_1}{j} f(x)^{p^{\ell_1}j} \lambda^{p^{\ell_1} (s_1-j)} + \lambda.
\end{align*}
By the condition 
$$
p^{\ell_2}(s_2-1) < p^{\ell_1}(s_1-1), 
$$
the highest $\lambda$-degree term with $\deg_x(c_{n,i})>0$ in the above formula is $s_1f(x)^{p^{\ell_1}} \lambda^{p^{\ell_1}(s_1-1)}$ so (c) is true when $n=2$.

Assume that the assertion (c) holds for some $n \ge 2$. Then the term of $f_{\lambda}^n(x)$ with the highest $\lambda$-degree is $\lambda^{d_2^{n-1}}$ by (a), while the term with the highest $\lambda$-degree among those having $\deg_x(c_{n,i})>0$ is $a_n(x) \lambda^{b_n}$.
The highest $\lambda$-degree term of 
$$
f_{\lambda}^{n}(x)^{d_2}
= (\lambda^{p^{\ell_2}d_2^{n-1}} + a_n(x)^{p^{\ell_2}} \lambda^{p^{\ell_2}b_n} + \cdots )^{s_2}
$$
whose coefficient has a positive degree in $x$ is
\begin{equation} \label{eq5_term1}
s_2(\lambda^{p^{\ell_2}d_2^{n-1}})^{s_2-1} \cdot a_n(x)^{p^{\ell_2}} \lambda^{p^{\ell_2}b_n}
= s_2a_n(x)^{p^{\ell_2}} \lambda^{b_{n+1}}
\end{equation}
and the highest $\lambda$-degree term of 
$$
f_{\lambda}^{n}(x)^{d_1}
= (\lambda^{p^{\ell_1}d_2^{n-1}} + a_n(x)^{p^{\ell_1}} \lambda^{p^{\ell_1}b_n} + \cdots )^{s_1}
$$
whose coefficient has a positive degree in $x$ is
\begin{equation*} \label{eq5_term2}
s_1(\lambda^{p^{\ell_1}d_2^{n-1}})^{s_1-1} \cdot a_n(x)^{p^{\ell_1}} \lambda^{p^{\ell_1}b_n}
= s_1 a_n(x)^{p^{\ell_1}} \lambda^{b'_{n+1}}
\end{equation*}
where $b'_{n+1} := p^{\ell_1}(d_2^{n-1}(s_1-1)+b_n)$. 

Now we prove the inequality $b'_{n+1} < b_{n+1}$. The relation
$$
b_{n+1} = p^{\ell_2}(d_2^{n-1}(s_2-1)+b_n) = d_2^n - p^{\ell_2}d_2^{n-1}+p^{\ell_2} b_n
$$
gives $d_2^n - b_{n+1} = p^{\ell_2}(d_2^{n-1}-b_n)$ so
$$
d_2^{n-1} - b_{n} = p^{(n-2)\ell_2}(d_2-b_2)=p^{(n-2)\ell_2}(d_2-d_1+p^{\ell_1}).
$$
This implies that
\begin{align*}
b_{n+1}-b'_{n+1}
& = d_2^{n-1}(d_2-d_1) -(p^{\ell_2}-p^{\ell_1})(d_2^{n-1}-b_n) \\
& = d_2^{n-1}(d_2-d_1) -(p^{\ell_2}-p^{\ell_1})p^{(n-2)\ell_2}(d_2-d_1+p^{\ell_1}) \\
& = (d_2^{n-1}-p^{(n-2)\ell_2}(p^{\ell_2}-p^{\ell_1}))(d_2-d_1) - p^{(n-2)\ell_2}(p^{\ell_2}-p^{\ell_1}) \cdot p^{\ell_1}.
\end{align*}
By the conditions $s_2>1$ and $n \ge 2$, we have 
$$
d_2^{n-1}-p^{(n-2)\ell_2}(p^{\ell_2}-p^{\ell_1}) > p^{(n-2)\ell_2}(p^{\ell_2}-p^{\ell_1})
$$
and
$$
d_2-d_1 = p^{\ell_1}(s_2p^{\ell_2-\ell_1}-s_1) \ge p^{\ell_1}
$$
so $b_{n+1}-b_{n+1}'>0$. 
This shows that the highest $\lambda$-degree term of $f_{\lambda}^{n+1}(x)$ satisfying $\deg_x(c_{n+1, i})>0$ has $\lambda$-degree $b_{n+1}$. 
Moreover, except for $s_2a_n(x)^{p^{\ell_2}}\lambda^{b_{n+1}}$ given by \eqref{eq5_term1}, every term of $f_{\lambda}^{n+1}(x)$ with $\lambda$-degree $b_{n+1}$ has $x$-degree $0$. Hence, the coefficient of $\lambda^{b_{n+1}}$ in $f_{\lambda}^{n+1}(x)$ is $s_2a_n(x)^{p^{\ell_2}} + u_n$ for some $u_n \in \olfp$.
\end{proof}

We present a few additional lemmas. Unlike the first one, the proofs of the remaining lemmas are similar to the corresponding results in \cite{AG25}. 

\begin{lem} \label{lem:53}
If $C(\al_1,\al_2;\be)$ is infinite, then $\al_1\in \ol{\fp(\al_2)}$ and $\al_2\in \ol{\fp(\al_1)}$.
\end{lem}

\begin{proof}
The proof is identical to that of \cite[Lemma 3.6]{AG25}, so we only give a brief sketch here. Let $L_1 := \ol{\fp(\al_2)}$ and $K := L_1(\al, \be)$. Assume that $\al_1 \notin L_1$, so $K/L_1$ is a function field either of transcendence degree $1$ or $2$. 
By the assumption $\al_1 \notin L_1$, there is a place $v$ of $K$ such that $\lt \al_1 \rt_v > 1$.

For $\lambda \in C(\al_1,\al_2;\be)$, there exist $m,n \in \bN$ such that $f_{\lambda}^m(\al_1) = f_{\lambda}^n(\al_2) =\be$. 
Since the set $C(\al_1,\al_2;\be)$ is infinite and \cite[Lemma 3.3]{AG25} holds for any polynomial $f$ of degree greater than $1$, we may assume that both $m$ and $n$ are sufficiently large so that $\lt \be \rt_v < \lt \al_1 \rt_v^{d_2^{\min(m,n)}}$.

By Lemma \ref{lem5:polynomialproperty}(a), 
$$
\lt c_{m,d_2^{m-1}} (\al_1) \rt_v= \lt f^m(\al_1) \rt_v = |\al_1|_v^{d_2^m} > \lt \be \rt_v = \lt f_{\lambda}^m(\al) \rt_v
$$
so 
$$
\lt c_{m,i}(\al_1) \lambda^{d_2^{m-1}-i}\rt_v\ge |\al_1|_v^{d_2^m}
$$
for some $0 \le i \le d_2^{m-1}-1$. By Lemma \ref{lem5:polynomialproperty}(b), 
$$
|c_{m,i}(\al_1)|_v\le |\al_1|_v^{d_2\cdot i}
$$
for each $i$ so 
$\lt \lambda \rt_v \ge \lt \al_1 \rt_v^{d_2}$. Now we have
$$
\lt \be \rt_v 
= \lt f_{\lambda}^n(\al_2) \rt_v 
\ge |\lambda|_v^{d_2^{n-1}} \ge |\al_1|_v^{d_2^n},
$$
which contradicts the assumption $\lt \be \rt_v < \lt \al_1 \rt_v^{d_2^{\min(m,n)}}$. We conclude that $\al_1 \in L_1$.
\end{proof}

Using the above result, we prove the following lemma analogous to \cite[Proposition 3.7]{AG25}. Note that if $C(\al_1,\al_2;\be)$ is infinite, then $\ol{\fp(\al_1)}=\ol{\fp(\al_2)}$ by Lemma \ref{lem:53}. 

\begin{lem}\label{lem:54} 
Let $\al_1, \al_2, \be \in L$ and assume that $C(\al_1,\al_2;\be)$ is infinite. Then at least one of the following statements holds:
\begin{enumerate}
    \item[(a)] $f(\al_1)=f(\al_2)$,
    \item[(b)] $\be \in L_1 := \ol{\fp(\al_1)}=\ol{\fp(\al_2)}$.
\end{enumerate}
\end{lem}

\begin{proof}
Assume that (b) is not true, so $\be \notin L_1$. Then
$K_1=L_1(\be)$ is a rational function field over $L_1$ of transcendence degree $1$. We regard $K_1$ as the function field of $\mathbb{P}_{L_1}^1$. 
Let $|\cdot |_\infty$ denote the absolute value on the function field $K_1/L_1$ associated with the place at infinity on $\mathbb{P}_{L_1}^1$ and fix an extension of $|\cdot |_\infty$ to $\ol{K_1}$. 

Let $\lambda \in C(\al_1,\al_2;\be)$ and $f_\lambda^m(\al_1)=f_\lambda^n(\al_2)=\be$ for some $m,n \in \bN$. By the equation $|f^m_{\lambda}(\al_1)|_\infty 
= |\be|_\infty > 1$ and $|\al_1|_\infty \le 1$, we have
$|\lambda|_\infty>1$. Then
$$
|\lambda|_\infty^{d_2^{m-1}} 
= |f^m_{\lambda}(\al_1)|_\infty 
= |f^n_{\lambda}(\al_2)|_\infty
= |\lambda|_\infty^{d_2^{n-1}}
$$
by \eqref{eq5_iteration} and Lemma \ref{lem5:polynomialproperty}(a). Since $|\lambda|_\infty>1$, the equality $|\lambda|_\infty^{d_2^{m-1}}=|\lambda|_\infty^{d_2^{n-1}}$ implies that $m=n$.

Now we have $f_\lambda^m(\al_1)=f_\lambda^m(\al_2)=\be$. The case $m=1$ is clear, so we may assume that $m \ge 2$.
Since $L_1$ is algebraically closed and $|\lambda|_\infty>1$, $\lambda$ is transcendental over $L_1$. Hence, the equality $f_\lambda^m(\al_1)=f_\lambda^m(\al_2)$ implies that the coefficients of $f_\lambda^m(\al_1)$ and $f_\lambda^m(\al_2)$ (as polynomials in $\lambda$) are equal. 
By Lemma \ref{lem5:polynomialproperty}(c), 
among the terms of $f_{\lambda}^m(x)$ whose coefficients $c_{m,i}(x)$ have positive degrees in $x$, the term with the highest degree in $\lambda$ is $a_m(x) \lambda^{b_m}$ where $a_m(x)$ and $b_m$ are given as in Lemma \ref{lem5:polynomialproperty}. Since $a_2(x)=s_1f(x)^{p^{\ell_1}}$ and
$$
a_{n+1}(x)=s_2a_n(x)^{p^{\ell_2}}+u_n \;\; (u_n \in \olfp)
$$
for each $n \ge 2$, an induction on $m$ shows that $a_m(\al_1) = a_m(\al_2)$ implies $f(\al_1) = f(\al_2)$.
\end{proof}

Next we provide two lemmas on the canonical heights of the polynomial $f_{\lambda}$.

\begin{lem}\label{lem:57}
Let $K$ be a finite extension field of $L_0:= \olfp(t,t^{1/p},t^{1/p^2},\ldots )$, $\al \in K$ and $\lambda \in \ol{K}$. Then
$$
\frac{h(\lambda)}{d_2} - h(\al)\le \widehat{h}_{f_\lambda}(\al)\le \frac{h(\lambda)}{d_2}+h(\al).
$$
\end{lem}

\begin{proof}
The proof is identical to that of \cite[Proposition 5.2]{AG25}, except that \cite[Lemma 5.1]{AG25} is replaced with \cite[Lemma 2.6]{GH24}. 
\end{proof}

\begin{lem} \label{lem:58}
Let $\al,\be \in K$, $\lambda \in \ol{K}$ and $m \in \bN$. If $f_\lambda^m(\al)=\be$, then
$$
\widehat{h}_{f_\lambda}(\al)\le \frac{2h(\al)+2h(\be)}{d_2^m}.
$$
\end{lem}

\begin{proof}
See the proof of \cite[Proposition 5.3]{AG25}. 
\end{proof}

\begin{proof}[Proof of the forward direction of Theorem \ref{thm1_colliding}]

Assume that $C(\al_1,\al_2;\be)$ is infinite and $f(\al_1) \ne f(\al_2)$. By Lemma \ref{lem:54}, $\be \in \ol{\fp(\al_1)}=\ol{\fp(\al_2)}$. If $\al_1 \in \olfp$, then $\al_1, \al_2, \be \in \olfp$ which is impossible by the assumption. Hence $\al_1 \notin \olfp$. 
Since $\al_1$ is transcendental over $\olfp$, we may take $t := \al_1$ in the field $L_0$. 

The condition $\be \in \ol{\fp(\al_1)} =\ol{\fp(\al_2)}$ ensures that there is a finite extension $K$ of $L_0$ containing $\al_1, \al_2, \be$. 
Following the proof of \cite[Proposition 6.2]{AG25} (using Lemma \ref{lem:58}), one can prove that there exists an infinite sequence $(\lambda_k)_{k=1}$ of distinct elements of $\ol{K}$ such that 
$$
\lim_{k \to \infty}\widehat{h}_{f_{\lambda_k}}(\al_1)
= \lim_{k \to \infty}\widehat{h}_{f_{\lambda_k}}(\al_2)=0.
$$
By Theorem \ref{thm:sameheight}, we have
$$
\widehat{h}_{v, \lambda}(\al_1) = \widehat{h}_{v, \lambda}(\al_2)
$$
for every $v \in \Omega_L$ and $\lambda \in \ol{L}$. By the proof of Theorem \ref{thm_main4}, this implies $f(\al_1)=f(\al_2)$. We conclude that the forward direction of Theorem \ref{thm1_colliding} holds.
\end{proof}

\subsection{Converse direction of Theorem \ref{thm1_colliding}} \label{ss:52}

The following theorem is an analogue of \cite[Theorem 2.4]{AG25}. It immediately implies that if $f(\al_1)=f(\al_2)$, then $C(\al_1,\al_2;\be)$ is infinite.

\begin{thm}\label{thm:57}
Let $L$ be a field of characteristic $p$ and $\al, \be \in L$. Under Setup \ref{setup_f}, there are infinitely many $\lambda \in \ol{L}$ such that $f_\lambda^m(\al)=\be$ for some $m \in \bN$.
\end{thm}

\begin{proof}
Assume that the set
$$
C(\al;\be) := \{ \lambda\in \ol{L}: \text{ there exists } m \in \bN \text{ such that } f_\lambda^m(\al)=\be \}
$$
is finite, i.e. $C(\al;\be)=\{ \lambda_1, \ldots , \lambda_r\}$.
For each $m \in \bN$, $f_\lambda^m(\al)$ can be viewed as a polynomial in $\lambda$ so
$$
f_u^m(\al)-\be=\prod_{i=1}^r(u-\lambda_i)^{e_{m,i}}
$$
for some $e_{m,1}, \ldots , e_{m,r} \ge 0$. By the relation $f_u^{m+1}(\al)=f_u^m(\al)^{d_2}+c_1f_u^m(\al)^{d_1}+u$, we have 
$$
\left(\prod_{i=1}^r(u-\lambda_i)^{e_{m,i}} +\be\right)^{d_2} +c_1\left(\prod_{i=1}^r(u-\lambda_i)^{e_{m,i}} +\be\right)^{d_1}+u-\be =\prod_{i=1}^r(u-\lambda_i)^{e_{m+1,i}}.
$$
Let $\Gamma$ be the subgroup of $\mathbb{G}_m^2(L(u))$ spanned by the elements $(u-\lambda_i,1)$ and $(1,u-\lambda_i)$ for $i=1,2,\ldots, r$, and let $V$ be a curve inside $\mathbb{G}_m^2$ defined by the equation
\begin{equation} \label{eq5_curve}
(x+\be)^{d_2}+c_1(x+\be)^{d_1}=y+(\be-u).
\end{equation}
The curve $V$ is geometrically irreducible since it is linear in $y$. Then $V(L(u))\cap \Gamma$ is infinite since it contains all points of the form
$$
(f_u^m(\al)-\be,f_u^{m+1}(\al)-\be)=\left( \prod_{i=1}^r (u-\lambda_i)^{e_{m,i}},\prod_{i=1}^r (u-\lambda_i)^{e_{m+1,i}}\right).
$$
The equation \eqref{eq5_curve} defining $V$ shows that the curve $V$ is not the translate of an algebraic subgroup of $\mathbb{G}_m^2$. Indeed, the equation of any translate of a proper subtorus of $\mathbb{G}_m^2$ (defined over $\ol{L(u)}$) is of the form $x^ay^b=\zeta$ for some $\zeta\in \ol{L(u)}^*$ and some integers $a$ and $b$, not both equal to $0$. Since $\be \in L$ and $u$ is transcendental over $\ol{L}$, $\be^{d_2}+c_1 \be^{d_1} \ne \be - u$ so $V$ does not contain any translate of a proper subtorus of $\mathbb{G}_m^2$.

Now we can conclude the proof by following the argument in \cite[Section 7.2]{AG25}. By the assumption $s_2>1$, $d_2$ is not a power of $p$ so the same argument applies.
\end{proof}

\section*{Acknowledgments}		
Jungin Lee was supported by the National Research Foundation of Korea (NRF) grant funded by the Korea government (MSIT) (No. RS-2024-00334558 and No. RS-2025-02262988). 
GyeongHyeon Nam was supported by the National Research Foundation of Korea (NRF) grant funded by the Korea government (MSIT) (No. RS-2024-00334558) and Oscar Kivinen's Väisälä project grant of the Finnish Academy of Science and Letters.
We are very grateful to Dragos Ghioca for informing us of his recent work on the colliding orbits problem and for providing helpful comments and suggestions.
	

\end{document}